\documentclass{article}
\usepackage{lmodern}
\usepackage[T1]{fontenc}
\usepackage[latin9]{inputenc}
\usepackage{units}
\usepackage{amsthm}
\usepackage{amsmath}
\usepackage{amssymb}
\usepackage{graphicx}
\usepackage{esint}

\makeatletter
\theoremstyle{plain}
\newtheorem{thm}{\protect\theoremname}
\theoremstyle{definition}
\newtheorem{defn}[thm]{\protect\definitionname}
\theoremstyle{plain}
\newtheorem{prop}[thm]{\protect\propositionname}
\theoremstyle{plain}
\newtheorem{lem}[thm]{\protect\lemmaname}
\ifx\proof\undefined
\newenvironment{proof}[1][\protect\proofname]{\par
\normalfont\topsep6\p@\@plus6\p@\relax
\trivlist
\itemindent\parindent
\item[\hskip\labelsep\scshape #1]\ignorespaces
}{%
\endtrivlist\@endpefalse
}
\providecommand{\proofname}{Proof}
\fi

\usepackage{witmse}
\usepackage{epsfig}

\usepackage{pgf,tikz}
\usepackage{mathrsfs}
\usetikzlibrary{arrows}

\title{PROPER SCORING AND SUFFICIENCY}
\name{Peter Harremo{\"e}s}
\address{Niels Brock, Copenhagen Business College,   \\
   Copenhagen, DENMARK, harremoes@ieee.org }
\makeatother

\providecommand{\definitionname}{Definition}
\providecommand{\lemmaname}{Lemma}
\providecommand{\propositionname}{Proposition}
\providecommand{\theoremname}{Theorem}

\begin{document}
\maketitle 
\begin{abstract}
Logarithmic score and information divergence appear in both information
theory, statistics, statistical mechanics, and portfolio theory. We
demonstrate that all these topics involve some kind of optimization
that leads directly to the use of Bregman divergences. If a sufficiency
condition is also fulfilled the Bregman divergence must be proportional
to information divergence. The sufficiency condition has quite different
consequences in the different areas of application, and often it is
not fulfilled. Therefore the sufficiency condition can be used to
explain when results from one area can be transferred directly from
one area to another and when one will experience differences.
\end{abstract}

\section{INTRODUCTION}

The use of scoring rules has a long history in statistics. An early
contribution was the idea of minimizing the sum of square deviations
that dates back to Gauss and works perfectly for Gaussian distributions.
In the 1920's Ramsay and de Finetti proved versions of the Dutch book
theorem where determination of probability distributions were considered
as dual problems to maximizing a payoff function. Later it was proved
that any consistent inference corresponds to optimizing with respect
to some payoff function. A more systematic study of scoring rules
was given by McCarthy \cite{McCarthy1956} and has recently been studied
by Dawid, Lauritzen and Parry \cite{Dawid2012} where the notion of
a local scoring rule has been extended. The basic result is that the
only strictly local proper scoring rule is logarithmic score.

Thermodynamics is the study of concepts like heat, temperature and
energy. A major objective is to extract as much energy from a system
as possible. Concepts like entropy and free energy play a significant
role. The idea in statistical mechanics is to view the macroscopic
behavior of a thermodynamic system as a statistical consequence of
the interaction between a lot of microscopic components where the
interacting between the components are governed by very simple laws.
Here the central limit theorem and large deviation theory play a major
role. One of the main achievements is the formula for entropy as a
logarithm of a probability.

One of the main purposes of information theory is to compress data
so that data can be recovered exactly or approximately. One of the
most important quantities was called entropy because it is calculated
according to a formula that mimics the calculation of entropy in statistical
mechanics. Another key concept in information theory is information
divergence (KL-divergence) that was introduced by Kullback and Leibler
in 1951 in a paper entitled information and sufficiency. The link
from information theory back to statistical physics was developed
by E.T. Jaynes via the maximum entropy principle. The link back to
statistics is now well established \cite{Liese1987,Barron1998,Csiszar2004}. 

The relation between information theory and gambling was established
by Kelly\cite{Kelly1956}. Logarithmic terms appear because we are
interested in the exponent in an exponential growth rate of of our
wealth. Later Kelly's approach has been generalized to training of
stocks although the relation to information theory is weaker \cite{Cover1991}.

Related quantities appear in statistics, statistical mechanics, information
theory and finance, annd we are interested in a theory that describes
when these relations are exact and when they just work by analogy.
First we introduce some general results about optimization on convex
sets. This part applies exactly to all the topics under consideration
and lead to Bregman divergences. Secondly, we introduce a notion of
sufficiency and show that this leads to information divergence and
logarithmic score. This second step is not always applicable which
explains when the different topics are really different.

Proofs of the theorems in this short paper can be found in an appendix
that is part of the arXiv version of the paper.

\section{STATE SPACE}

The present notion of a state space is based on \cite{Holevo1982},
and is mainly relevant for quantum systems. 

Before we do anything we prepare our system. Let $\mathcal{P}$ denote
the set of preparations. Let $p_{0}$ and $p_{1}$ denote two preparations.
For $t\in\left[0,1\right]$ we define $\left(1-t\right)\cdot p_{0}+t\cdot p_{1}$
as the preparation obtained by preparing $p_{0}$ with probability
$1-t$ and $t$ with probability t. A measurement $m$ is defined
as an affine mapping of the set of preparations into a set of probability
measures on some measurable space. Let $\mathcal{M}$ denote a set
of feasible measurements. The state space $\mathcal{S}$ is defined
as the set of preparations modulo measurements. Thus, if $p_{1}$
and $p_{2}$ are preparations then they represent the same state if
$m\left(p_{1}\right)=m\left(p_{2}\right)$ for any $m\in\mathcal{M}.$

In statistics the state space equals the set of preparations and has
the shape of a simplex. The symmetry group of a simplex is simply
the group of permutations of the extreme points. In quantum theory
the state space has the shape of the density matrices on a complex
Hilbert space and the state space has a lot of symmetries that a simplex
does not have. For simplicity we will assume that the state space
is a finite dimensional convex compact space.

\section{OPTIMIZATION}

Let $\mathcal{A}$ denote a subset of the feasible measurements $\mathcal{M}$
such that $a\in A$ maps $S$ into a distribution on the real numbers
i.e. a random variable. The elements of $\mathcal{A}$ may represent
actions like the score of a statistical decision, the energy extracted
by a certain interaction with the system, (minus) the length of a
codeword of the next encoded input letter using a specific code book,
or the revenue of using a certain portfolio. For each $s\in\mathcal{S}$
we define $F\left(s\right)=\sup_{a\in\mathcal{A}}E\left[a\left(s\right)\right]$.
We note that $F$ is convex but $F$ need not be strictly convex.
We say that a sequence of actions $\left(a_{n}\right)_{n}$ is \emph{asymptotically
optimal} for the state $s$ if $E\left[a_{n}\left(s\right)\right]\to F\left(s\right)$
for $n\to\infty.$

If the state is $s_{1}$ but one acts as if the state were $s_{2}$
one suffers a \emph{regret} that equals the difference between what
one achieves and what could have been achieved. 
\begin{defn}
If $F\left(s_{1}\right)$ is finite \emph{the regret} is defined by
\begin{equation}
D_{F}\left(s_{1},s_{2}\right)=F\left(s_{1}\right)-\sup_{\left(a_{n}\right)_{n}}\limsup_{n\to\infty}E\left[a_{n}\left(s_{1}\right)\right]
\end{equation}
where the supremum is taken over all sequences $\left(a_{n}\right)_{n}$
that are asymptotically optimal over $s_{2}.$ \end{defn}
\begin{prop}
The regret $D_{F}$ has the following properties:\end{prop}
\begin{itemize}
\item $D_{F}\left(s_{1},s_{2}\right)\geq0$ with equality if $s_{1}=s_{2}.$
\item $\sum t_{i}\cdot D_{F}\left(s_{i},\tilde{s}\right)\geq\sum t_{i}\cdot D_{F}\left(s_{i},\hat{s}\right)+D_{F}\left(\hat{s},\tilde{s}\right)$
where $\left(t_{1},t_{2},\dots,t_{\ell}\right)$ is a probability
vector and $\hat{s}=\sum t_{i}\cdot s_{i}$.
\item $\sum t_{i}\cdot D_{F}\left(s_{i},\tilde{s}\right)$ is minimal when
$\hat{s}=\sum t_{i}\cdot s_{i}$.
\end{itemize}
If the state space is finite dimensional and there exists a unique
action $a_{2}$ such that $F\left(s_{2}\right)=E\left[a\left(s_{2}\right)\right]$
then $D_{F}\left(s_{1},s_{2}\right)=E\left[a_{1}\left(s_{1}\right)\right]-E\left[a_{2}\left(s_{1}\right)\right]$.
If unique optimal actions exists for any state then $F$ is differentiable
which implies that the regret can be written as a \emph{Bregman divergence}
in the following form
\begin{align}
D_{F}\left(s_{1},s_{2}\right) & =F\left(s_{1}\right)-\left(F\left(s_{2}\right)+\left\langle s_{1}-s_{2},\nabla F\left(s_{2}\right)\right\rangle \right).
\end{align}
In the context of forecasting and statistical scoring rules the use
of Bregman divergences dates back to \cite{Hendrickson1971}.

We note that $D_{F_{1}}\left(s_{1},s_{2}\right)=D_{F_{2}}\left(s_{1},s_{2}\right)$
if and only if $F_{1}\left(s\right)-F_{2}\left(s\right)$ is an affine
function of $s.$ If the state $s_{2}$ has the unique optimal action
$a_{2}$ then 
\begin{equation}
F\left(s_{1}\right)=D_{F}\left(s_{1},s_{2}\right)+E\left[a_{2}\left(s_{1}\right)\right]
\end{equation}
 so the function $F$ can be reconstructed from $D_{F}$ except for
an affine function of $s_{1}.$ The closure of the convex hull of
the set of functions $s\to E\left[a\left(s\right)\right]$ is uniquely
determined by the convex function $F.$

\section{SUFFICIENCY}

Let $\left(s_{\theta}\right)_{\theta}$ denote a family of states
and let $\Phi$ denote \emph{a completely positive} transformation
$\mathcal{S}\to\mathcal{T}$ where $\mathcal{S}$ and $\mathcal{T}$
denote state spaces. Then $\Phi$ is said to be \emph{sufficient}
for $\left(s_{\theta}\right)_{\theta}$ if there exists a completely
positive transformation $\Psi:\mathcal{T}\to\mathcal{S}$ such that
$\Psi\left(\Phi\left(s_{\theta}\right)\right)=s_{\theta}.$ 

We say that the regret $D_{F}$ on the state space $S$ satisfies
the \emph{sufficiency property} if $D_{F}\left(\Phi\left(s_{1}\right),\Phi\left(s_{2}\right)\right)=D_{F}\left(s_{1},s_{2}\right)$
for any completely positive transformation $\mathcal{S}\to\mathcal{S}$
that is sufficient for $\left(s_{1},s_{2}\right).$ The notion of
sufficiency as a property of divergences was introduced in \cite{Harremoes2007a}.
The crucial idea of restricting the attention to transformations of
the state space into itself was introduced in \cite{Jiao2014}.
\begin{thm}
Assume that $S$ is a state space. If the divergence $D_{F}$ satisfies
the sufficiency property then for any state $s$ and any completely
positive transformation $\Phi:S\to S$ one has $F\left(\Phi\left(s\right)\right)=F\left(s\right)$
.
\end{thm}
If the alphabet size is two the above condition on $F$ is sufficient
to conclude that 
\begin{equation}
D_{F}\left(\Phi\left(s_{1}\right),\Phi\left(s_{2}\right)\right)=D_{F}\left(s_{1},s_{2}\right).
\end{equation}

\begin{thm}
\label{thm:Generelt}Assume that the state space $S$ is a classical
or quantum state space on three or more letters. If the regret $D_{F}$
satisfies the sufficiency property, then $F$ is proportional to the
entropy function and $D_{F}$ is proportional to information divergence
(relative entropy).
\end{thm}
This theorem can be proved via a numer of partial results as explained
in the next section.

\section{APPLICATIONS}

\subsection{Statistics}

Consider an experiment with $\mathcal{X}=\left\{ 1,2,\dots,\ell\right\} $
as sample space. A \emph{scoring rule} $f$ is defined as a function
with domain $\mathcal{X}\times M_{1}^{+}\left(\mathcal{X}\right)\to R$
such that the score is $f\left(x,Q\right)$ when the prediction was
given by $Q$ and $x\in\mathcal{X}$ has been observed. A scoring
rule is \emph{proper} if for any probability measure $P\in M_{1}^{+}\left(\mathcal{X}\right)$
the score $\sum_{x\in\mathcal{X}}P\left(x\right)\cdot f\left(x,Q\right)$
is minimal when $Q=P.$
\begin{thm}
\label{thm:Smooth}The scoring rule $f$ is proper is and only if
there exists a smooth function $F$ such that $f\left(x,Q\right)=D_{F}\left(\delta_{x},Q\right)+\tilde{f}\left(x\right).$\end{thm}
\begin{defn}
A \emph{strictly local scoring rule} is a scoring rule of the form
$f\left(x,Q\right)=g\left(Q\left(x\right)\right).$\end{defn}
\begin{lem}
\label{lem:SuffImplLocal}On a finite space a Bregman divergence that
satisfies the sufficiency condition gives a strictly local scoring
rule. 
\end{lem}
The following theorem was given in \cite{Jiao2014} with a much longer
proof.
\begin{thm}
\label{thm:JiaoVenkat}On a finite alphabet with at least three letters
a Bregman divergence that satisfies the sufficiency condition is proportional
to information divergence. \end{thm}
\begin{proof}
Since any strictly local proper scoring rule corresponds to separable
divergence a divergence that is Bregman and satisfies sufficiency
must also be separable. If the alphabet size is at least three the
only separable divergences that are Bregman divergences are the ones
proportional to information divergence \cite{Harremoes2007a}. 
\end{proof}

\subsection{Information theory}

Let $b_{1},b_{2},\dots,b_{n}$ denote the letters of an alphabet and
let $\ell\left(\kappa\left(b_{i}\right)\right)$ denote the length
of the codeword $\kappa\left(b_{i}\right)$ according to some code
book $\kappa.$ If the code is uniquely decodable then $\sum2^{-\ell\left(\kappa\left(b_{i}\right)\right)}\leq1.$
Note that $\ell\left(\kappa\left(b_{i}\right)\right)$ is an integer.
If only integer values of $\ell$ are allowed then $h$ is piecewise
linear and sufficiency is not fulfilled. If arbitrary real numbers
are allowed then it obvious we get a proper local scoring rule.

\subsection{Statistical mechanics}

Statistical mechanics can be stated based on classical mechanics or
quantum mechanics. For our purpose this makes no difference because
Theorem \ref{thm:Generelt} can be applied for both classical systems
and quantum systems.
\begin{proof}[Proof of Theorem \ref{thm:Generelt}]
 If we restrict to any commutative sub-algebra the divergence is
proportional to information divergence as stated in Theorem \ref{thm:JiaoVenkat}
so that $F$ is proportional to the entropy function $H$ restricted
to the sub-algebra. Any state generates a commutative sub-algebra
so the function $F$ is proportional to $H$ on all states and the
divergence is proportional to information divergence.
\end{proof}
Assume that a heat bath of temperature $T$ is given and that all
the states are close to the state of the heat bath. An action $a\in A$
is some interaction with the thermodynamic system that extracts some
energy from the system. In thermodynamics the quantity $F\left(s\right)=\sup_{a\in A}E\left[a\left(s\right)\right]$
is normally called the free energy. If the temperature is kept fixed
under all interactions $F$ is called \emph{Helmholtz free energy}.
Any sufficient transformation $\Phi$ for $s_{1}$ and $s_{2}$ is
quasi-static and can be approximately realized by a physical process
$\Psi$ that is reversible in the thermodynamic sense of the word.

\begin{align}
D_{F}\left(\Phi\left(s_{1}\right),\Phi\left(s_{2}\right)\right) & =a_{\Phi\left(s_{1}\right)}\left(\Phi\left(s_{1}\right)\right)-a_{\Phi\left(s_{2}\right)}\left(\Phi\left(s_{1}\right)\right).
\end{align}
Now 
\begin{multline}
a_{\Phi\left(s_{2}\right)}\left(\Phi\left(s_{2}\right)\right)=\left(a_{\Phi\left(s_{2}\right)}\circ\Phi\right)\left(s_{2}\right)\\
\leq a_{2}\left(s_{2}\right)=a_{2}\left(\Psi\left(\Phi\left(s_{2}\right)\right)\right)\\
=\left(a_{2}\circ\Psi\right)\left(\Phi\left(s_{2}\right)\right)\leq a_{\Phi\left(s_{2}\right)}\left(\Phi\left(s_{2}\right)\right).
\end{multline}
Hence $a_{\Phi\left(s_{2}\right)}=a_{2}\circ\Psi$ so that
\begin{multline}
D_{F}\left(\Phi\left(s_{1}\right),\Phi\left(s_{2}\right)\right)\\
=\left(a_{1}\circ\Psi\right)\left(\Phi\left(s_{1}\right)\right)-\left(a_{2}\circ\Psi\right)\left(\Phi\left(s_{1}\right)\right)\\
=a_{1}\left(s_{1}\right)-a_{2}\left(s_{1}\right)=D_{F}\left(s_{1},s_{2}\right).
\end{multline}
The amount of extractable energy $Ex$ is proportional to information
divergence. The quotient between extractable energy and information
divergence depends on the temperature and one may even define the
absolute temperature via the formula 
\begin{equation}
Ex=kT\cdot D\left(s_{1}\left\Vert s_{2}\right.\right)\label{eq:Kelvin}
\end{equation}
where $k=1.381\cdot10^{-23}\nicefrac{\unit{J}}{\unit{K}}$ is Boltzmann's
constant. Equation (\ref{eq:Kelvin}) was derived already in \cite{Harremoes1993}
by a similar argument.

According to Equation (\ref{eq:Kelvin}) any bit of information can
be converted into an amount of energy! One may ask how this is related
to the mixing paradox (a special case of Gibbs' paradox). Consider
a container divided by a wall with a blue and a yellow gas on each
side of the wall. The question is how much energy can be extracted
by mixing the gasses?

\begin{center}
\includegraphics{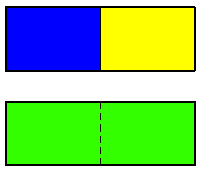}
\par\end{center}

We loose one bit of information about each molecule by mixing the
gasses, but if the color is the \emph{only difference} no energy can
be extracted. This seems to be in conflict with Equation (\ref{eq:Kelvin}),
but in this case different states cannot be converted into each other
by reversible processes. For instance one cannot convert the blue
gas into the yellow gas. To get around this problem one can restrict
the set of preparations and one can restrict the set of measurements.
For instance one may simply ignore measurements of the color of the
gas. What should be taken into account and what should be ignored,
can only be answered by an experienced physicist. Formally this solves
the mixing paradox but from a practical point of view nothing has
been solved. If for instance the molecules in one of the gasses are
much larger than the molecules in the other gas then a semi-permeable
membrane can be used to create an osmotic pressure that can be used
to extract some energy. It is still an open question which differences
in properties of the two gasses that can be used to extract energy.

\subsection{Portfolio theory}

Let $X_{1},X_{2},\dots,X_{k}$ denote \emph{price relatives} for a
list of stocks. For instance $X_{5}=1.04$ means that stock no. 5
increases its value by 4 \%. A \emph{portfolio} is a probability vector
$\vec{b}=\left(b_{1},b_{2},\dots,b_{k}\right)$ where for instance
$b_{5}=0.3$ means that 30 \% of your money is invested in stock no.
5. The total price relative is $X_{1}\cdot b_{1}+X_{2}\cdot b_{2}+\dots+X_{k}\cdot b_{k}=\vec{X}\cdot\vec{b}.$
We now consider a situation where the stocks are traded once every
day. For a sequence of price relative vectors $\vec{X}_{1},\vec{X_{2}},\dots\vec{X}_{n}$
and \emph{a constant re-balancing portfolio} $\vec{b}$ the wealth
after $n$ days is
\begin{equation}
S_{n}=\prod_{i=1}^{n}\left\langle \vec{X}_{i},\vec{b}\right\rangle 
\end{equation}
According to law of large numbers 
\begin{equation}
\frac{1}{n}\log\left(S_{n}\right)\to E\left[\log\left\langle \vec{X},\vec{b}\right\rangle \right]
\end{equation}
Here $E\left[\log\left\langle \vec{X},\vec{b}\right\rangle \right]$
is proportional to the \emph{doubling rate} and is denoted $W\left(\vec{b},P\right)$
where $P$ indicates the probability distribution of $\vec{X}$. Our
goal to maximize $W\left(\vec{b},P\right)$ by choosing an appropriate
portfolio $\vec{b}.$

Let $\vec{b}_{P}$ denote the portfolio that is optimal for $P$.
As proved in \cite{Cover1991}
\begin{equation}
W\left(\vec{b}_{P},P\right)-W\left(\vec{b}_{Q},P\right)\leq D\left(\left.P\right\Vert Q\right).
\end{equation}

\begin{thm}
\label{thm:portfolio}The Bregman divergence 
\begin{equation}
W\left(\vec{b}_{P},P\right)-W\left(\vec{b}_{Q},P\right)
\end{equation}
 satisfies the equation 
\begin{equation}
W\left(\vec{b}_{P},P\right)-W\left(\vec{b}_{Q},P\right)=D\left(\left.P\right\Vert Q\right).
\end{equation}
if and only if the measure $P$ on $k$ distinct vectors of the form
$\left(a_{1},0,0,\dots0\right)$, $\left(0,a_{2},0,\dots0\right),$
until $\left(0,0,\dots a_{k}\right).$ 
\end{thm}

\section{CONCLUSION}

On the level of optimization the theory works out in exactly the same
way in statistics, information theory, statistical mechanics, and
portfolio theory. The sufficiency condition is more complicated to
apply. It requires that we restrict to a certain class of mappings
of the state space into itself. In the case where the state space
can be identified with a set of density matrices one should restrict
to completely positive maps. In case the state space has a different
structure it is not obvious which mappings one should restrict to.
The basic problem is that we have to introduce a notion of tensor
product for convex sets and it is not obvious how to do this, but
this will be the topic of further investigations and results on this
topic may have some impact on our general understanding of quantum
theory.

The original paper of Kullback and Leibler \cite{Kullback1951} was
called ``On Information and Sufficiency''. In the present paper
we have made the relation between information divergence and the notion
of sufficiency more explicit. The idea of sufficiency has different
consequences in different applications but in all cases information
divergence prove to be the quantity that convert the general notion
of sufficiency into a number. For specific applications one cannot
identify the sufficient variables without studying the specific application
in detail. For problems like the the mixing paradox there is still
no simple answer to the question about what the sufficient variables
are, but if the sufficient variables have been specified we have the
mathematical framework to develop the rest of the theory in a consistent
manner.


\begin{thebibliography}{10}
\providecommand{\url}[1]{#1}
\csname url@samestyle\endcsname
\providecommand{\newblock}{\relax}
\providecommand{\bibinfo}[2]{#2}
\providecommand{\BIBentrySTDinterwordspacing}{\spaceskip=0pt\relax}
\providecommand{\BIBentryALTinterwordstretchfactor}{4}
\providecommand{\BIBentryALTinterwordspacing}{\spaceskip=\fontdimen2\font plus
\BIBentryALTinterwordstretchfactor\fontdimen3\font minus
  \fontdimen4\font\relax}
\providecommand{\BIBforeignlanguage}[2]{{%
\expandafter\ifx\csname l@#1\endcsname\relax
\typeout{** WARNING: IEEEtran.bst: No hyphenation pattern has been}%
\typeout{** loaded for the language `#1'. Using the pattern for}%
\typeout{** the default language instead.}%
\else
\language=\csname l@#1\endcsname
\fi
#2}}
\providecommand{\BIBdecl}{\relax}
\BIBdecl

\bibitem{McCarthy1956}
J.~{McCarthy}, ``Measures of the value of information,'' \emph{Proc. Nat. Acad.
  Sci.}, vol.~42, pp. 654--655, 1956.

\bibitem{Dawid2012}
A.~P. Dawid, S.~Lauritzen, and M.~Parry, ``Proper local scoring rules on
  discrete sample spaces,'' \emph{The Annals of Statistics}, vol.~40, no.~1,
  pp. 593--608, 2012.

\bibitem{Liese1987}
F.~Liese and I.~Vajda, \emph{Convex Statistical Distances}.\hskip 1em plus
  0.5em minus 0.4em\relax Leipzig: Teubner, 1987.

\bibitem{Barron1998}
A.~R. Barron, J.~Rissanen, and B.~Yu, ``The minimum description length
  principle in coding and modeling,'' \emph{IEEE Trans. Inform. Theory},
  vol.~44, no.~6, pp. 2743--2760, Oct. 1998, commemorative issue.

\bibitem{Csiszar2004}
I.~Csisz{\'a}r and P.~Shields, \emph{Information Theory and Statistics: A
  Tutorial}, ser. Foundations and Trends in Communications and Information
  Theory.\hskip 1em plus 0.5em minus 0.4em\relax Now Publishers Inc., 2004.

\bibitem{Kelly1956}
J.~L. Kelly, ``A new interpretation of information rate,'' \emph{Bell System
  Technical Journal}, vol.~35, pp. 917--926, 1956.

\bibitem{Cover1991}
T.~Cover and J.~A. Thomas, \emph{Elements of Information Theory}.\hskip 1em
  plus 0.5em minus 0.4em\relax Wiley, 1991.

\bibitem{Holevo1982}
A.~S. Holevo, \emph{Probabilistic and Statistical Aspects of Quantum Theory},
  ser. North-Holland Series in Statistics and Probability, P.~R. Krishnaiah,
  C.~R. Rao, M.~Rosenblatt, and Y.~A. Rozanov, Eds.\hskip 1em plus 0.5em minus
  0.4em\relax Amsterdam: North-Holland, 1982, vol.~1.

\bibitem{Hendrickson1971}
A.~D. Hendrickson and R.~J. Buehler, ``Proper scores for probability
  forecasters,'' \emph{Ann. Math. Statist.}, vol.~42, pp. 1916--1921, 1971.

\bibitem{Harremoes2007a}
\BIBentryALTinterwordspacing
P.~Harremo{\"e}s and N.~Tishby, ``The information bottleneck revisited or how
  to choose a good distortion measure,'' in \emph{Proceedings ISIT 2007,
  Nice}.\hskip 1em plus 0.5em minus 0.4em\relax IEEE Information Theory
  Society, June 2007, pp. 566--571. [Online]. Available:
  \url{www.harremoes.dk/Peter/flaske2.pdf}
\BIBentrySTDinterwordspacing

\bibitem{Jiao2014}
J.~Jiao, T.~C. amd Albert~No, K.~Venkat, and T.~Weissman, ``Information
  measures: the curious case of the binary alphabet,'' \emph{Trans. Inform.
  Theory}, vol.~60, no.~12, pp. 7616--7626, Dec. 2014.

\bibitem{Harremoes1993}
P.~Harremo{\"e}s, \emph{Time and Conditional Independence}, ser.
  IMFUFA-tekst.\hskip 1em plus 0.5em minus 0.4em\relax IMFUFA Roskilde
  University, 1993, vol. 255, original in Danish entitled Tid og Betinget
  Uafh{\ae}ngighed. English translation partially available.

\bibitem{Kullback1951}
S.~Kullback and R.~Leibler, ``On information and sufficiency,'' \emph{Ann.
  Math. Statist.}, vol.~22, pp. 79--86, 1951.

\end{thebibliography}

\newpage{}

\section*{Appendix}

\subsection*{Proof of Theorem \ref{thm:Smooth}}

If $f$ is given in terms of a regret function $D_{F}$ then
\begin{multline}
\sum_{x\in\mathcal{X}}P\left(x\right)\cdot f\left(x,Q\right)\\
=\sum_{x\in\mathcal{X}}P\left(x\right)\cdot\left(D_{F}\left(\delta_{x},Q\right)+g\left(x\right)\right)\\
\geq\sum_{x\in\mathcal{X}}P\left(x\right)\cdot D_{F}\left(\delta_{x},P\right)+D_{F}\left(P,Q\right)\\
\,\,+\sum_{x\in\mathcal{X}}P\left(x\right)\cdot g\left(x\right)
\end{multline}
because $P=\sum_{x\in\mathcal{X}}P\left(x\right)\cdot\delta_{x}$.
If $F$ is smooth then $D_{F}\left(P,Q\right)=0$ if and only if $Q=P.$

Assume that $f$ is proper. Then we may define a divergence by
\begin{equation}
D\left(P,Q\right)=\sum_{x\in\mathcal{X}}P\left(x\right)\cdot f\left(x,Q\right)-\sum_{x\in\mathcal{X}}P\left(x\right)\cdot f\left(x,P\right).
\end{equation}
Since $f$ is assumed to be proper $D\left(P,Q\right)\geq0$ with
equality if and only if $P=Q.$ The equality $\sum t_{i}\cdot D_{F}\left(P_{i},Q\right)=\sum t_{i}\cdot D_{F}\left(P_{i},\hat{P}\right)+D_{F}\left(\hat{P},Q\right)$
follows by straight forward calculations. With these two results we
see that $D$ equals a Bregman divergence $D_{F}$ and that 
\begin{multline}
D_{F}\left(\delta_{y},Q\right)\\
=\sum_{x\in\mathcal{X}}\delta_{y}\left(x\right)\cdot f\left(x,Q\right)-\sum_{x\in\mathcal{X}}\delta_{y}\left(x\right)\cdot f\left(x,\delta_{y}\right)\\
=f\left(y,Q\right)-f\left(y,\delta_{y}\right).
\end{multline}
Hence $f\left(x,Q\right)=D_{F}\left(\delta_{x},Q\right)+f\left(x,\delta_{x}\right).$

\subsection*{Proof of Lemma \ref{lem:SuffImplLocal}}

Let $D_{F}$ denote a regret function that satisfies the sufficiency
condition. Then
\begin{multline}
D_{F}\left(\delta_{i},\left(q_{1},q_{2},\dots,q_{\ell}\right)\right)\\
=D_{F}\left(\delta_{1},\left(q_{i},q_{i+1},\dots,q_{i-2},q_{i-1}\right)\right)
\end{multline}
where we have made a cyclic permutation of indices. Next we use the
sufficient transformation that projects a mixture of $\delta_{1}$
and a uniform distribution. 
\begin{multline}
D_{F}\left(\delta_{i},\left(q_{1},q_{2},\dots,q_{\ell}\right)\right)\\
=D_{F}\left(\delta_{1},\left(q_{i},\frac{\sum_{j\neq i}q_{j}}{\ell-1},\frac{\sum_{j\neq i}q_{j}}{\ell-1},\dots,\frac{\sum_{j\neq i}q_{j}}{\ell-1}\right)\right)\\
=D_{F}\left(\delta_{1},\left(q_{i},\frac{1-q_{i}}{\ell-1},\frac{1-q_{i}}{\ell-1},\dots,\frac{1-q_{i}}{\ell-1}\right)\right).
\end{multline}
Note that the projection can be obtained by taking a mixture of all
permutations of the extreme points that leave the first extreme point
unchanged. Hence the scoring rule is given by the local scoring rule
\[
g\left(p\right)=D_{F}\left(\delta_{1},\left(q,\frac{1-q}{\ell-1},\frac{1-q}{\ell-1},\dots,\frac{1-q}{\ell-1}\right)\right).
\]

\subsection*{Proof of Theorem \ref{thm:portfolio}}

We have 
\begin{multline}
W\left(\vec{b}_{P},P\right)-W\left(\vec{b}_{Q},P\right)\\
=\int\log\left\langle \vec{X},\vec{b_{P}}\right\rangle \,\mathrm{d}PX-\int\log\left\langle \vec{X},\vec{b_{Q}}\right\rangle \,\mathrm{d}PX\\
=\int\log\left(\frac{\left\langle \vec{X},\vec{b_{P}}\right\rangle }{\left\langle \vec{X},\vec{b_{Q}}\right\rangle }\right)\,\mathrm{d}PX\\
=\int\log\left(\frac{\left\langle \vec{X},\vec{b_{P}}\right\rangle }{\left\langle \vec{X},\vec{b_{Q}}\right\rangle }\cdot\frac{\mathrm{d}Q}{\mathrm{d}P}\cdot\frac{\mathrm{d}P}{\mathrm{d}Q}\right)\,\mathrm{d}PX\\
=\int\log\left(\frac{\left\langle \vec{X},\vec{b_{P}}\right\rangle }{\left\langle \vec{X},\vec{b_{Q}}\right\rangle }\cdot\frac{\mathrm{d}Q}{\mathrm{d}P}\right)\,\mathrm{d}PX+D\left(P\Vert Q\right).
\end{multline}
Next we use Jensen's inequality to get
\begin{multline}
W\left(\vec{b}_{P},P\right)-W\left(\vec{b}_{Q},P\right)\\
\leq\log\left(\int\frac{\left\langle \vec{X},\vec{b_{P}}\right\rangle }{\left\langle \vec{X},\vec{b_{Q}}\right\rangle }\cdot\frac{\mathrm{d}Q}{\mathrm{d}P}\,\mathrm{d}PX\right)+D\left(P\Vert Q\right)\\
=\log\left(\int\frac{\left\langle \vec{X},\vec{b_{P}}\right\rangle }{\left\langle \vec{X},\vec{b_{Q}}\right\rangle }\,\mathrm{d}QX\right)+D\left(P\Vert Q\right)\\
\leq\log\left(1\right)+D\left(P\Vert Q\right)=D\left(P\Vert Q\right).
\end{multline}
Jensen's inequality holds with equality if and only if 
\begin{equation}
\frac{\left\langle \vec{X},\vec{b_{P}}\right\rangle }{\left\langle \vec{X},\vec{b_{Q}}\right\rangle }\cdot\frac{\mathrm{d}Q}{\mathrm{d}P}
\end{equation}
is constant $P$-almost surely. Equivalently $\frac{dQ}{dP}$ is proportional
to $\frac{\left\langle \vec{X},\vec{b_{Q}}\right\rangle }{\left\langle \vec{X},\vec{b_{P}}\right\rangle }$
for any probability measure $Q$ on the support of $\vec{X}.$ The
set of vectors $\vec{b_{Q}}$ lie in a $k-1$ dimensional convex set.
Therefore the set of probability measures on the support of $P$ is
at most $k-1$ dimensional. Hence $P$ is supported on at most $k$
vectors in $\mathbb{\mathbf{\mathbb{R}}}_{0,+}^{k}$.

The inequality 
\begin{equation}
\int\frac{\left\langle \vec{X},\vec{b_{P}}\right\rangle }{\left\langle \vec{X},\vec{b_{Q}}\right\rangle }\,\mathrm{d}QX\leq1
\end{equation}
 holds with equality if $\left(\vec{b_{Q}}\right)_{i}=0$ implies
$\left(\vec{b_{P}}\right)_{i}=0.$ If $\frac{\left\langle \vec{X},\vec{b_{P}}\right\rangle }{\left\langle \vec{X},\vec{b_{Q}}\right\rangle }=k\cdot\frac{\mathrm{d}P}{\mathrm{d}Q}$
we have 
\begin{eqnarray}
\int\frac{\left\langle \vec{X},\vec{b_{P}}\right\rangle }{\left\langle \vec{X},\vec{b_{Q}}\right\rangle }\,\mathrm{d}QX & = & \int k\cdot\frac{\mathrm{d}P}{\mathrm{d}Q}\,\mathrm{d}QX=k
\end{eqnarray}
so $\frac{\left\langle \vec{X},\vec{b_{P}}\right\rangle }{\left\langle \vec{X},\vec{b_{Q}}\right\rangle }\leq\frac{\mathrm{d}P}{\mathrm{d}Q}$.
The reversed inequality is proved in the same way so we get 
\begin{equation}
\frac{\mathrm{d}P}{\mathrm{d}Q}=\frac{\left\langle \vec{X},\vec{b_{P}}\right\rangle }{\left\langle \vec{X},\vec{b_{Q}}\right\rangle }.\label{eq:EnTilEn}
\end{equation}
Equation (\ref{eq:EnTilEn}) gives an affine bijection between distributions
$P$ and portfolios $\vec{b_{P}}$. The set of portfolios is a simplex
with $k$ extreme points so the set of distributions must also be
a simplex with $k$ extreme points. Therefore support of the probability
measures is a set of $k$ vectors. We denote the vector of price relatives
corresponding to $\vec{b_{\delta_{j}}}$ by $\vec{x_{j}}$ and the
$i$'th coordinate of this vector by $\left(\vec{x_{j}}\right)_{i}$. 

The portfolio $\vec{b}=\left(b_{1},b_{2},\dots,b_{k}\right)$ is optimal
for the probability distribution with weight $b_{j}$ on the vector
$\vec{x_{j}}.$ According to the Kuhn-Tucker conditions \cite[Thm. 15.2.1]{Cover1991}
the vector $\vec{b}=\left(b_{1},b_{2},\dots,b_{k}\right)$ is optimal
for the probability distribution with weight $b_{j}$ on the vector
$\vec{x_{j}}$ if 
\begin{equation}
\sum_{j=1}^{k}\frac{\left(\vec{x_{j}}\right)_{i}}{\left\langle \vec{b},\vec{x_{j}}\right\rangle }\cdot b_{j}\leq1
\end{equation}
 with equality for all $i$ for which $b_{i}>0.$ Assume that $b_{j}=\delta_{\ell}\left(j\right).$
Then we get the inequality 
\begin{equation}
\frac{\left(\vec{x_{\ell}}\right)_{i}}{\left(\vec{x_{\ell}}\right)_{\ell}}\leq1
\end{equation}
or, equivalently, $\left(\vec{x_{\ell}}\right)_{i}\leq\left(\vec{x_{\ell}}\right)_{\ell}$. 

If $b_{\ell}=s>0$ and $b_{m}=t>0$ where $s+t=1$ then 
\begin{eqnarray}
1 & = & \frac{\left(\vec{x_{\ell}}\right)_{\ell}}{\left\langle \vec{b},\vec{x_{\ell}}\right\rangle }\cdot b_{\ell}+\frac{\left(\vec{x_{m}}\right)_{\ell}}{\left\langle \vec{b},\vec{x_{m}}\right\rangle }\cdot b_{m}\\
 & = & \frac{\left(\vec{x_{\ell}}\right)_{\ell}\cdot b_{\ell}}{b_{\ell}\cdot\left(\vec{x_{\ell}}\right)_{\ell}+b_{m}\cdot\left(\vec{x_{\ell}}\right)_{m}}\\
 &  & +\frac{\left(\vec{x_{m}}\right)_{\ell}\cdot b_{m}}{b_{\ell}\cdot\left(\vec{x_{m}}\right)_{\ell}+b_{m}\cdot\left(\vec{x_{m}}\right)_{m}}.
\end{eqnarray}
Hence 
\begin{equation}
\frac{b_{m}\cdot\left(\vec{x_{\ell}}\right)_{m}}{b_{\ell}\cdot\left(\vec{x_{\ell}}\right)_{\ell}+b_{m}\cdot\left(\vec{x_{\ell}}\right)_{m}}=\frac{\left(\vec{x_{m}}\right)_{\ell}\cdot b_{m}}{b_{\ell}\cdot\left(\vec{x_{m}}\right)_{\ell}+b_{m}\cdot\left(\vec{x_{m}}\right)_{m}}
\end{equation}
and
\begin{equation}
\frac{\left(\vec{x_{\ell}}\right)_{m}}{b_{\ell}\cdot\left(\vec{x_{\ell}}\right)_{\ell}+b_{m}\cdot\left(\vec{x_{\ell}}\right)_{m}}=\frac{\left(\vec{x_{m}}\right)_{\ell}}{b_{\ell}\cdot\left(\vec{x_{m}}\right)_{\ell}+b_{m}\cdot\left(\vec{x_{m}}\right)_{m}}
\end{equation}
which is equivalent to 
\begin{multline}
\left(\vec{x_{\ell}}\right)_{m}\left(b_{\ell}\cdot\left(\vec{x_{m}}\right)_{\ell}+b_{m}\cdot\left(\vec{x_{m}}\right)_{m}\right)\\
=\left(\vec{x_{m}}\right)_{\ell}\left(b_{\ell}\cdot\left(\vec{x_{\ell}}\right)_{\ell}+b_{m}\cdot\left(\vec{x_{\ell}}\right)_{m}\right).
\end{multline}
This should hold for all positive $b_{\ell},b_{m}$ for which $b_{\ell}+b_{m}=1$
so it also holds for the limiting value $b_{m}=0$ where the equality
reduces to 
\begin{equation}
\left(\vec{x_{\ell}}\right)_{m}\left(\vec{x_{m}}\right)_{\ell}=\left(\vec{x_{m}}\right)_{\ell}\left(\vec{x_{\ell}}\right)_{\ell}
\end{equation}
 so that either $\left(\vec{x_{m}}\right)_{\ell}=0$ or $\left(\vec{x_{\ell}}\right)_{m}=\left(\vec{x_{\ell}}\right)_{\ell}$.
Similarly we get $\left(\vec{x_{\ell}}\right)_{m}=0$ or $\left(\vec{x_{m}}\right)_{\ell}=\left(\vec{x_{m}}\right)_{m}$.
Together we get either $\left(\vec{x_{m}}\right)_{\ell}=0$ and $\left(\vec{x_{\ell}}\right)_{m}=0$,
or $\left(\vec{x_{\ell}}\right)_{m}=\left(\vec{x_{\ell}}\right)_{\ell}$
and $\left(\vec{x_{m}}\right)_{\ell}=\left(\vec{x_{m}}\right)_{m}.$
Therefore $\left(\vec{x_{m}}\right)_{\ell}=0$ or $\left(\vec{x_{m}}\right)_{\ell}=\left(\vec{x_{m}}\right)_{m}$.

Let $\sim$ denote the relation on $\left\{ 1,2,3,\dots,k\right\} $
defined by $\ell\sim m$ when $\left(\vec{x_{m}}\right)_{\ell}=\left(\vec{x_{m}}\right)_{m}$.
The relation $\sim$ is obviously reflexive, and as we have seen it
is symmetric. We will prove that $\sim$ is transitive. Assume that
$\ell\sim m$ and $m\sim n$. Then $\left(\vec{x_{\ell}}\right)_{\ell}=\left(\vec{x_{\ell}}\right)_{m}$
and $\left(\vec{x_{m}}\right)_{\ell}=\left(\vec{x_{m}}\right)_{m}$=$\left(\vec{x_{m}}\right)_{n}$
and $\left(\vec{x_{n}}\right)_{m}=\left(\vec{x_{n}}\right)_{n}.$
Assume further that $\ell\nsim n$ so that $\left(\vec{x_{\ell}}\right)_{n}=\left(\vec{x_{n}}\right)_{\ell}=0.$
Assume that $b_{\ell}=s>0,$ $b_{m}=t>0,$and $b_{n}=u>0,$ and $s+t+u=1$
\begin{eqnarray}
1 & = & \frac{\left(\vec{x_{\ell}}\right)_{n}}{s\left(\vec{x_{\ell}}\right)_{\ell}+t\left(\vec{x_{\ell}}\right)_{m}+u\left(\vec{x_{\ell}}\right)_{n}}\cdot s\\
 &  & +\frac{\left(\vec{x_{m}}\right)_{n}}{s\left(\vec{x_{m}}\right)_{\ell}+t\left(\vec{x_{m}}\right)_{m}+u\left(\vec{x_{m}}\right)_{n}}\cdot t\\
 &  & +\frac{\left(\vec{x_{n}}\right)_{n}}{s\left(\vec{x_{n}}\right)_{\ell}+t\left(\vec{x_{n}}\right)_{m}+u\left(\vec{x_{n}}\right)_{n}}\cdot u\\
1 & = & \frac{0}{s\left(\vec{x_{\ell}}\right)_{\ell}+t\left(\vec{x_{\ell}}\right)_{\ell}}\cdot s\\
 &  & +\frac{\left(\vec{x_{m}}\right)_{m}}{s\left(\vec{x_{m}}\right)_{m}+t\left(\vec{x_{m}}\right)_{m}+u\left(\vec{x_{m}}\right)_{m}}\cdot t\\
 &  & +\frac{\left(\vec{x_{n}}\right)_{n}}{t\left(\vec{x_{n}}\right)_{n}+u\left(\vec{x_{n}}\right)_{n}}\cdot u\\
1 & = & t+\frac{u}{t+u}.
\end{eqnarray}
This should hold for all $s,t,u$ which is a contradiction. Therefore
$\ell\sim n$ and we conclude that $\sim$ is transitive. 

Since $\sim$ is transitive either $\vec{x_{\ell}}$ and $\vec{x_{m}}$
are orthogonal or they are parallel with price relatives that are
either zero or have the same price relatives that are the same for
stock $\ell$ and stock $m$. Therefore we may consider stock $\ell$
and stock $m$ as the same stock. Hence we may exclude the case where
vectors are parallel, so all the vectors are orthogonal but this is
only possible if the vectors are proportional to the basis vectors.
\end{document}